\documentclass[11pt,leqno]{article}

\usepackage{amsmath, amsthm, amsfonts, amssymb, stmaryrd,color}
\usepackage{mathrsfs}
\newtheorem{thm}{Theorem}[section]

\newtheorem{lem}[thm]{Lemma}

\theoremstyle{definition}

\numberwithin{equation}{section} \theoremstyle{remark}
\newtheorem{rem}{Remark}[section]

\title{{\bf Weak convergence of path-dependent SDEs with irregular coefficients}
\thanks{Supported in
 part by NNSFs of China (Nos. 11771327, 11301030, 11431014, 11831014).} }
\author{
{\bf  Jianhai Bao$^{b,c)}$,  Jinghai Shao$^{a)}$}\\
\footnotesize{$^{a)}$Center for Applied Mathematics, Tianjin
University, Tianjin 300072, China}\\
\footnotesize{$^{b)}$School of Mathematics and Statistics, Central
South
University, Changsha 410083, China}\\
\footnotesize{$^{c)}$Department of Mathematics, Swansea University,
Singleton Park, SA2 8PP, UK}\\ \footnotesize{
jianhaibao@csu.edu.cn~~~~ shaojh@bnu.edu.cn }}

\begin{document}
\def\R{\mathbb R}  \def\ff{\frac} \def\ss{\sqrt} \def\B{\mathbf
B}
\def\N{\mathbb N} \def\kk{\kappa} \def\m{{\bf m}}
\def\dd{\delta} \def\DD{\Delta} \def\vv{\varepsilon} \def\rr{\rho}
\def\<{\langle} \def\>{\rangle} \def\GG{\Gamma} \def\gg{\gamma}
  \def\nn{\nabla} \def\pp{\partial} \def\EE{\scr E}
\def\d{\text{\rm{d}}} \def\bb{\beta} \def\aa{\alpha} \def\D{\scr D}
  \def\si{\sigma} \def\ess{\text{\rm{ess}}}
\def\beg{\begin} \def\beq{\begin{equation}}  \def\F{\scr F}
\def\Ric{\text{\rm{Ric}}} \def\Hess{\text{\rm{Hess}}}
\def\e{\text{\rm{e}}} \def\ua{\underline a} \def\OO{\Omega}  \def\oo{\omega}
 \def\tt{\tilde} \def\Ric{\text{\rm{Ric}}}
\def\cut{\text{\rm{cut}}} \def\P{\mathbb P} \def\ifn{I_n(f^{\bigotimes n})}
\def\C{\scr C}      \def\aaa{\mathbf{r}}     \def\r{r}
\def\gap{\text{\rm{gap}}} \def\prr{\pi_{{\bf m},\varrho}}  \def\r{\mathbf r}
\def\Z{\mathbb Z} \def\vrr{\varrho} \def\ll{\lambda}
\def\L{\scr L}\def\Tt{\tt} \def\TT{\tt}\def\II{\mathbb I}
\def\i{{\rm in}}\def\Sect{{\rm Sect}}\def\E{\mathbb E} \def\H{\mathbb H}
\def\M{\scr M}\def\Q{\mathbb Q} \def\texto{\text{o}} \def\LL{\Lambda}
\def\Rank{{\rm Rank}} \def\B{\scr B} \def\i{{\rm i}} \def\HR{\hat{\R}^d}
\def\to{\rightarrow}\def\l{\ell}
\def\8{\infty}\def\X{\mathbb{X}}\def\3{\triangle}
\def\V{\mathbb{V}}\def\M{\mathbb{M}}\def\W{\mathbb{W}}\def\Y{\mathbb{Y}}\def\1{\lesssim}

\def\La{\Lambda}\def\S{\mathbf{S}}

\renewcommand{\bar}{\overline}
\renewcommand{\hat}{\widehat}
\renewcommand{\tilde}{\widetilde}
\newcommand{\scr}[1]{\mathscr #1}
\newcommand{\lra}{\longrightarrow}
\def\ua{\uparrow}
\newcommand{\da}{\downarrow}
\newcommand{\dsp}{\displaystyle}
 \maketitle

\begin{abstract}
In this paper we develop via Girsanov's transformation a
perturbation argument to investigate weak convergence of
Euler-Maruyama (EM) scheme for path-dependent SDEs with  H\"older
continuous drifts. This approach is available to other scenarios,
e.g., truncated EM schemes for non-degenerate SDEs with finite
memory or infinite memory. Also, such trick can be applied to study
weak convergence  of truncated EM scheme for a range of
 stochastic Hamiltonian systems with irregular coefficients and with
 memory, which  are typical degenerate dynamical systems. Moreover, the weak convergence of path-dependent SDEs
 under integrability condition is investigated by establishing,
via the dimension-free Harnack inequality,   exponential
integrability of irregular drifts w.r.t.
 the invariant probability measure constructed explicitly in
 advance.

\end{abstract}
 \vskip 2cm

\section{Introduction}
The strong/weak convergence of numerical schemes for SDEs with
 regular coefficients has been investigated extensively; see, e.g.,
\cite{BT,KP1,KP2,GR,NT1,NT,PT,TT} and reference therein.  Also, the weak convergence
for SDEs with irregular terms has gained much attention; see,
e.g., \cite{KP,MP} with the payoff function being smooth.
For path-dependent SDEs
(which, in terminology, are also named as functional SDEs or SDEs
with delays),
there is
considerable literature on strong convergence of various numerical
schemes (e.g.,  truncated/tamed EM scheme) under regular conditions; see, for instance,
\cite{DKS,Mao15} and references within.
In contrast, weak convergence analysis of numerical algorithms for path-dependent SDEs is scarce. The path-dependent SDEs under irregular conditions are much more difficult than SDEs under irregular conditions. This work deals with the weak approximation of numerical algorithms for path-dependent SDEs under H\"older continuity condition or certain integrability condition.

As far as path-dependnet SDEs are concerned, the weak convergence of
numerical methods was initiated in \cite{KP} whereas the rigorous
justification of their statements was unavailable. With regard to
weak convergence of EM scheme and its variants,  we refer to
\cite{BS} for a class of semi-linear path-dependnet SDEs   via the
so-called ``lift-up'' approach, \cite{CKL} for path-dependnet SDEs
with distributed delays by means of the duality trick, and
\cite{BKMS} for path-dependent SDEs with point delays with the help
of Malliavin calculus and the tamed It\^o formula. In the references
\cite{BKMS,CKL}, as for the drift term $b$ and the diffusion term
$\si$, the assumptions that $b,\si\in C_b^\8(\R^d)$ and the payoff
function $f\in C_b^3(\R^d)$ were imposed. Subsequently, by the aid
of Malliavin calculus, \cite{Zhang} extended \cite{BKMS,CKL} in a
certain sense that the payoff function $f\in\mathscr{B}_b(\R^d)$
while $b,\si\in C_b^\8(\R^d)$ therein. It is worthy of pointing out
that the approaches adopted in \cite{BKMS,CKL,Zhang} are applicable
merely for path-dependent SDEs with regular coefficients. In the
literature \cite{BKMS,Zhang}, the tamed It\^o formula plays a
crucial role in investigating weak convergence of EM scheme.
Nevertheless, the tamed It\^o formula works barely for
path-dependent SDEs with distributed delays or point delays so that
it seems hard to extend \cite{BKMS,Zhang} to path-dependent SDEs
with general delays.

To study weak convergence of numerical schemes
for path-independent SDEs with regular coefficients, the approach on
the Kolmogorov backward equation is one of the powerful methods.
However, concerning path-dependent SDEs, the Kolmogrov backward
equation is in general unavailable so that it cannot be adopted to
deal with weak convergence of numerical schemes. As we stated above,
concerning path-dependent SDEs, the Malliavin calculus is the
effective tool to cope with weak convergence; see, for example,
\cite{BKMS,CKL,Zhang}. Whereas,   a little bit strong assumptions
are imposed therein and the proof is not succinct in certain sense.
Moreover, Zvonkin's trasformation \cite{AZ} is one of the powerful
tools to investigate  strong convergence of EM schemes for
path-independent SDEs with singular coefficients; see, e.g.,
\cite{PT}. Nevertheless, such trick no longer works for path-dependent SDEs provided the appearance of the
delay terms.

In this
work we aim to develop a perturbation approach to study weak
convergence of (truncated) EM scheme for path-dependent SDEs with additive
noise, which allows the drift terms to be irregular (e.g.,
H\"older continuous drifts and integrability drifts) and even the
diffusion coefficients to be degenerate.
Elaborate estimation of the growth of a stochastic process under H\"older continuity condition and the application of the
dimension-free Harnack inequality under integrability condition play the crucial role in current work.

The content of this paper is arranged as follows. In Section
\ref{sec1}, we investigate weak convergence of EM scheme for a class
of non-degenerate SDEs with memory and reveal the weak convergence
rate; In Section \ref{sec3}, we apply the approach adopted in
Section \ref{sec1} to other scenarios, e.g., truncated EM scheme for
non-degenerate SDEs with finite memory or infinite memory; In
Section \ref{sec4}, we focus on weak convergence and reveal the weak
convergence order of truncated EM scheme for a range of stochastic
Hamiltonian systems with singular drifts and with memory; In the
last section, we are interested in weak convergence of EM scheme for
path-dependent SDEs under integrability conditions, which allow  the
drift terms to be singular.

 Before proceeding further, a few
words about the notation are in order. Throughout this paper, $c>0$
stands for a generic constant which might change from occurrence to
occurrence and depend on the time parameters.

\section{Weak Convergence: Non-degenerate Case }\label{sec1}
Let $(\R^d,\<\cdot,\cdot\>,|\cdot|)$ be the $d$-dimensional
Euclidean space with the inner product $\<\cdot,\cdot\>$ which
induces the norm $|\cdot|$. Let $\mathbb{M}_{\mathrm{non}}^{d}$   be the
set of all non-singular $d\times d$-matrices with real entries,
equipped with the Hilbert-Schmidt norm $\|\cdot\|_{\rm HS}$. $A^*$
means the transpose of the matrix $A.$ For a sub-interval
$\mathbb{U}\subseteq\R$, denote $ C(\mathbb{U};\R^d)$ by the family
of all continuous functions $f:\mathbb{U}\rightarrow\R^d$. Let
$\tau>0$ be a fixed number and
  $\C=C([-\tau,0];\R^d)$, which is endowed with the uniform norm
 $\|f\|_\8:=\sup_{-\tau\le
\theta\le 0}|f(\theta)|$. For   $f\in C([-\tau,\8);\R^d)$ and  fixed
$t\ge0$, let $f_t\in \C$ be defined by $f_t(\theta)=f(t+\theta),
\theta\in[-\tau,0].$ In terminology, $(f_t)_{t\ge0}$ is called the
segment (or window) process corresponding to $(f(t))_{t\ge-\tau}$.
For $a\ge0,$ $\lfloor a\rfloor$ stipulates the integer part of $a.$
Let $\B_b(\R^d)$ be the collection of all bounded measurable
functions $f:\R^d\to\R$, endowed with the norm $\interleave  f
\interleave_\8:=\sup_{x\in\R^d}|f(x)|$. Let ${\bf0}\in\R^d$ be the
zero vector and $\xi_0(\theta)\equiv{\bf 0}$ for any
$\theta\in[-\tau,0].$

In this  section, we are interested in the following path-dependent
SDE
\begin{equation}\label{eq1}
\d X(t)=\{b\,(X(t))+Z(X_t)\}\d t+\si\,\d
W(t),~~~t>0,~~~X_0=\xi\in\C,
\end{equation}
where $b:\R^d\to\R^d$, $Z:\C\rightarrow\R^d$,
$\si\in\mathbb{M}_{\mathrm{non}}^{d}$ and $(W(t))_{t\ge0}$ is a
$d$-dimensional Brownian motion on the probability space
$(\OO,\F,(\F_t)_{t\ge0},\P)$. We assume that
\begin{enumerate}
\item[({\bf A1})] $b$ is Lipschitz with the Lipschitz constant $L_1$, i.e.,
 $ |b(x)-b(y)|\le L_1|x-y|,~x,y\in\R^d, $ and there exist constants
  $C>0$ and $\bb\in\R$ such that
 \begin{equation}\label{eq13}
2\<x,b(x)\>\le C+\bb|x|^2,~~~~x\in\R^d;
 \end{equation}

\item[({\bf A2})] $Z $ is H\"older continuous with the H\"older exponent
$\aa\in(0,1]$ and the H\"older constant $L_2$, i.e., $
|Z(\xi)-Z(\eta)|\le L_2\|\xi-\eta\|^\aa_\8,~\xi,\eta\in\C; $

\item[({\bf A3})] The initial value $\xi\in\C$ is Lipschitz
continuous with the Lipschitz constant $L_3>0$, i.e., $
|\xi(t)-\xi(s)|\le L_3|t-s|,~s,t\in[-\tau,0]. $

\end{enumerate}
Under ({\bf A1}) and ({\bf A2}),  \eqref{eq1} enjoys a unique weak
solution $(X^\xi(t))_{t\ge0}$ with the initial datum
$X_0^\xi=\xi\in\C;$ see Lemma \ref{weak} below for more details.
Evidently, \eqref{eq13} holds with $\beta>0$ whenever $b$ obeys the
global Lipschitz condition. It is worthy to emphasize that $\beta$
in \eqref{eq13} need not to be positive, which may allow  the time
horizontal $T$ to be much bigger as Lemma \ref{novikov} below
manifests. Moreover,   ({\bf A3}) is just imposed for the sake of
continuity of the displacement of segment process. For further
details, please refer to Lemma \ref{con} below.

For existence and uniqueness of strong solutions to path-dependent
SDEs with regular coefficients, we refer to e.g. \cite{M08,M84,VS}
and references therein. Recently, path-dependnet SDEs with irregular
coefficients have also received much attention; see e.g. \cite{Bach}
on existence and uniqueness of strong solutions, \cite{Bachm} upon
strong Feller property of the semigroup generated by the functional
solution (i.e., the segment process associated with the solution
process), and \cite{Wang} about  the regularity estimates for the
density of invariant probability measures.

 Let $\dd\in(0,1)$ be  the stepsize given by
$\dd=\tau/M$ for some $M\in\mathbb{N}$ sufficiently large. Given the
stepsize $\dd\in(0,1)$, the continuous-time EM scheme associated
with \eqref{eq1} is defined as below
\begin{equation}\label{eq2}
\d X^{(\dd)}(t)=\{b(X^{(\dd)}(t_\dd))+Z(\hat X_{t_\dd}^{(\dd)})\}\d
t+\si\,\d W(t),~~~t>0
\end{equation}
with the initial value $X^{(\dd)}(\theta)=X(\theta)=\xi(\theta),
\theta\in[-\tau,0].$ Herein,
 $t_\dd:=\lfloor t/\dd\rfloor\dd$ and,   for any $k\in\mathbb{N},$ $\hat X_{k\dd}^{(\dd)}\in\C$ is  defined by
\begin{equation}\label{w6}
\hat
X_{k\dd}^{(\dd)}(\theta)=\ff{\theta+(1+i)\dd}{\dd}X^{(\dd)}((k-i)\dd)-\ff{\theta+i\dd}{\dd}
X^{(\dd)}((k-i-1)\dd)
\end{equation}
whenever $\theta\in[-(i+1)\dd,-i\dd]$ for
$i\in\S:=\{0,1,\cdots,M-1\}$, that is, the $\C$-valued process
$(\hat X_{k\dd}^{(\dd)})_{k\in\mathbb{N}}$ is constructed by the
linear interpolations between the points on the gridpoints.

To cope with the weak convergence of  EM scheme \eqref{eq2} with the
singular coefficient $Z$, in this work we shall adopt a perturbation
approach; see e.g. \cite{Wanga,Wang} on regularity estimates of
density of invariant probability measures for SDEs under
integrability conditions. To achieve this goal, we introduce the
following reference SDE on $\R^d$
\begin{equation}\label{eq3}
\d Y(t)=b(Y(t))\d t+\si\,\d W(t),~~~t>0,~~~Y(0)=x\in\R^d.
\end{equation}
Under ({\bf A1}), \eqref{eq3} has a unique strong solution
$(Y^x(t))_{t\ge0}$ with the initial value $Y(0)=x$; see, for
example,  \cite[Theorem 2.1, p34]{M84}. Now, let's extend $Y^x(t)$,
solving \eqref{eq3}, from $[0,\8)$ into $[-\tau,\8)$ in the
  manner below:
\begin{equation}\label{a6}
Y^\xi(t):=\xi(t){\bf 1}_{[-\tau,0)}(t)+Y^{\xi(0)}(t){\bf
1}_{[0,\8)}(t),~~~~t\in[-\tau,\8),~~~\xi\in\C.
\end{equation}
Let $(Y^\xi_t)_{t\ge0}$ be the segment process corresponding to
$(Y^\xi(t))_{t\ge-\tau}$.

Our  main result in this section is stated as follows, which in
particular reveals the weak convergence rate of EM algorithm
\eqref{eq2} associated with \eqref{eq1}, which allows the drift term
to be   H\"older continuous.

\begin{thm}\label{th1}
Let $({\bf A1})$, $({\bf A2})$ and $({\bf A3})$ hold. Then, for any
$\kappa\in(0,\aa/2)$   and $T>0$ such that
\begin{equation}\label{eq19}
2\,\|\si\|_{\rm HS}^2 \|\si^{-1}\|^2_{\rm HS}\{(4L_1^2+ L_2^2){\bf
1}_{\{\aa=1\}}+L_1^2{\bf1}_{\{\aa\in(0,1)\}}\}< \e^{-(1+\beta
T)}/T^2,
\end{equation}
there exists a constant $C_{1,T}>0$ such that
\begin{equation}\label{w10}
|\E f(X(t))-\E f(X^{(\dd)}(t))|\le
C_{1,T}\,\dd^{\kk},~~~~f\in\B_b(\R^d),~~t\in[0,T].
\end{equation}
\end{thm}

\begin{rem}
{\rm For  the path-independent SDE \eqref{eq1} with H\"older
continuous drift, \cite{NT} revealed the weak convergence order is
$\ff{\aa}{2}\wedge\ff{1}{4}$, where   $\aa\in(0,1)$ is the H\"older
exponent. Whereas, in Theorem \ref{th1}, we demonstrate that  the
weak convergence rate is $\aa/2.$ So Theorem \ref{th1} is new even
for path-independent SDEs with irregular drifts. For path-dependent
SDEs with point delays or distributed delays, \cite{BKMS,CKL}
investigated
 the weak convergence   under the regular assumption $Z\in C_b^\8$ and with the payoff function $f\in
C_b^3$. Nevertheless, in the present work, we might allow the drift
$Z$ to be   unbounded  and even    H\"older continuous  and most
importantly the payoff function $f$ to be   non-smooth. Hence,
Theorem \ref{th1} improves e.g. \cite{BKMS,CKL,NT,Zhang} in a
certain sense. Last but not least, the approached adopted to prove
Theorem \ref{th1} is universal in a sense that it is applicable to
the other scenarios as show in the Sections \ref{sec3} and
\ref{sec4}. }
\end{rem}

Before we move forward to complete the proof of Theorem \ref{th1},
let's prepare some warm-up lemmas. The following lemma address
existence and uniqueness of weak solutions to \eqref{eq1}.

\begin{lem}\label{weak}
{\rm Under  ({\bf A1}) and ({\bf A2}), \eqref{eq1} admits a unique
weak solution. }
\end{lem}

\begin{proof}
First of all, we show existence of a weak solution to \eqref{eq1}.
Set
\begin{equation*}
R^\xi_1(t):=\exp\Big(\int_0^t\<\si^{-1}Z(Y_s^\xi),\d
W(s)\>-\ff{1}{2}\int_0^t|\si^{-1}Z(Y_s^\xi)|^2\d s\Big),~~~t\ge0,
\end{equation*}
and $\d\mathbb{Q}^\xi_1:=R^\xi_1(T)\d \P$, where $T>0$ satisfies $
 \|\si\|_{\rm HS}^2 \|\si^{-1}\|_{\rm HS}^2L_2^2< \e^{-(1+\beta
T)}/T^2 $ for the setup of the H\"older exponent $\aa=1$ and $T>0$
is arbitrary with $\aa\in(0,1).$ Moreover, let
\begin{equation}\label{w22}
 W_1^\xi(t) =W(t)-\int_0^t\si^{-1}Z(Y_s^\xi)\d s,~~~~t\ge0.
\end{equation}
 According to Lemma \ref{novikov} below, we infer
that
\begin{equation*}
\E\,\e^{\ff{1}{2}\int_0^T|\si^{-1}Z(Y_t^\xi)|^2\d t}<\8,
\end{equation*}
that is, the Novikov condition holds true. Thus the Girsanov theorem
implies that $(W_1^\xi(t))_{t\in[0,T]}$ is a Brownian motion under
the weighted probability measure $\mathbb{Q}^\xi_1.$ Note that
\eqref{eq3} can be reformulated as
\begin{equation*}
\d Y^\xi(t)  =\{b(Y^\xi(t))+Z(Y_t^\xi)\}\d t+\si\,\d
W_1^\xi(t),~~~t\in[0,T],~~~Y_0^\xi=\xi.
\end{equation*}
So $(Y^\xi(t),W_1^\xi(t))_{t\in[0,T]}$ is a weak solution to
\eqref{eq1} under the probability space
$(\OO,\F,(\F_t)_{t\ge0},\mathbb{Q}^\xi_1)$. Analogously, we can show
inductively that \eqref{eq1} admits a weak solution on $[T,2T],
[2T,3T], \cdots$. Hence, \eqref{eq1} admits a global weak solution.

Now we proceed to justify  uniqueness of weak solutions to
\eqref{eq1}. In the sequel, it is sufficient to show the weak
uniqueness on the time interval $[0,T]$ since it can be done
analogously on $[T,2T], [2T,3T], \cdots$.
 Let $(X^{(i),\xi}(t), W^{(i)}(t))_{t\in[0,T]}$ be the
weak solution to \eqref{eq1} under the  probability space
$(\OO^{(i)},\F^{(i)},(\F_t^{(i)})_{t\ge0},\P_i^\xi), i=1,2.$ In
terms of \cite[Proposition 2.1, p169, \& Corollary, p206]{IW},   it
remains to show that
\begin{equation}\label{w4}
\E_{\P^\xi_1}f(X^{(1),\xi}([0,T]),W^{(1)}([0,T]))=\E_{\P^\xi_2}f(X^{(2),\xi}([0,T]),W^{(2)}([0,T]))
\end{equation}
for any $f\in C_b(C([0,T];\R^d)\times C([0,T];\R^d);\R)$, where
$\E_{\P^\xi_i}$ means the expectation w.r.t.  $\P^\xi_i$. Whereas
\eqref{w4} can be done exactly by following  the argument of
\cite[Theorem 2.1 (2)]{Wang}. We therefore complete the proof.
\end{proof}

The lemma below examines the exponential integrability of
functionals for segment process.

\begin{lem}\label{novikov}
{\rm Assume that ({\bf A1}) holds. Then, for any $T>0$,
\begin{equation}\label{eq11}
\E\,\e^{\ll\int_0^T\|Y_t^\xi\|^2_\8\d
t}<\8,~~~~\ll<\ff{\e^{-(1+\beta T)}}{2\,\|\si\|_{\rm HS}^2 T^2}.
\end{equation}
}
\end{lem}

\begin{proof}
Applying   Jensen's inequality and using the fact that
$\|Y_t^\xi\|_\8\le \|\xi\|_\8\vee\sup_{0\le s\le t}|Y^\xi(s)|$, we
have
\begin{equation}\label{s5}
\E\,\e^{\ll\int_0^T\|Y_t^\xi\|^2_\8\d t}   \le \ff{1}{T}
\int_0^T\E\,\e^{\ll T \,\|Y_t^\xi\|_\8^2}\d t\le\ff{\e^{\ll
T\|\xi\|_\8^2}}{T}\int_0^T\E\Big(\sup_{0\le s\le t}\e^{\ll
T|Y^{\xi(0)}(s)|^2}\Big)\d t,~ T>0.
\end{equation}
Next,   by It\^o's formula,     it follows from ({\bf A1})  that
\begin{equation}\label{eq7}
\begin{split}
&\d(\e^{-\gg t} |Y^{\xi(0)}(t)|^2)\\
&=\e^{-\gg t}\{-\gg  |Y^{\xi(0)}(t)|^2 +2\<Y^{\xi(0)}(t),b(Y^{\xi(0)}(t))\>+\|\si\|_{\rm HS}^2\}\d t\\
&\quad+2\,\e^{-\gg t}\<\si^*Y^{\xi(0)}(t),\d W(t)\>\\
&\le\e^{-\gg t}\{c-(\gg -\bb)|Y^{\xi(0)}(t)|^2\}\d t+2\,\e^{-\gg
t}\<\si^*Y^{\xi(0)}(t),\d W(t)\>,~~~~\gg>0.
\end{split}
\end{equation}
Also, via It\^o's formula,
  we deduce from \eqref{eq7} that
\begin{equation}\label{eq8}
\begin{split}
\d\e^{\vv\,\e^{- \gg t}|Y^{\xi(0)}(t)|^2}
&\le-\vv(\gg-\bb-2\|\si\|_{\rm HS}^2\vv)  \,\e^{-\gg
t}\e^{\vv\,\e^{-\gg t}|Y^{\xi(0)}(t)|^2}|Y^{\xi(0)}(t)|^2\d
t\\
&\quad+c\, \e^{\vv\,\e^{-\gg t}|Y^{\xi(0)}(t)|^2}\d t\\
&\quad+2\,\vv\,\e^{- \gg t}\,\e^{\vv\,\e^{-\gg
t}|Y^{\xi(0)}(t)|^2}\<\si^*Y^{\xi(0)}(t), \d
W(t)\>,~~~~~\gg>0,~~~\vv>0,
\end{split}
\end{equation}
which implies that, for any $\gg>\beta+2\|\si\|_{\rm HS}^2\vv,$ by
Gronwall's inequality,
\begin{equation}\label{s4}
\E\,\e^{\vv\,\e^{- \gg
t}|Y^{\xi(0)}(t)|^2}\le\e^{c\,t}\e^{\vv\,(1+|\xi(0)|^2)}.
\end{equation}
so that
\begin{equation}\label{s2}
\vv(\gg-\bb-2\|\si\|_{\rm HS}^2\vv)\int_0^t\e^{-\gg
s}\E(\e^{\vv\,\e^{-\gg s}|Y^{\xi(0)}(s)|^2} |Y^{\xi(0)}(s)|^2) \d
s\le (1+\e^{c\,t})\e^{\vv\,(1+|\xi(0)|^2)}.
\end{equation}
 Making use of  BDG's inequality and   Jensen's inequality, we derive
from   \eqref{eq8} and \eqref{s4}
 that
\begin{equation}\label{eq9}
\begin{split}
&\E\Big(\sup_{0\le s\le t}\e^{\vv\,\e^{-\gg
s}|Y^{\xi(0)}(s)|^2}\Big)\\&\le
\e^{\vv\,(1+|\xi(0)|^2)} +c\int_0^t \E\,\e^{\vv\,\e^{-\gg s}|Y^{\xi(0)}(s)|^2}\d s\\
&\quad+2\,\vv\E\big(\sup_{0\le s\le t}\int_0^s\e^{-\gg
u}\e^{\vv\,\e^{-\gg\,
u}(1+|Y^{\xi(0)}(u)|^2)}\<\si^*Y^{\xi(0)}(u),\d
W(u)\>\Big)\\
&\le(1+\e^{c\,t})\e^{\vv\,(1+|\xi(0)|^2)}\\
&\quad+8\ss2\,\vv\E\Big(\int_0^t\e^{-2\gg
s}\e^{2\vv\,\e^{-\gg\, s}(1+|Y^{\xi(0)}(s)|^2)}|\si^*Y^{\xi(0)}(s)|^2\d s\Big)^{1/2}\\
&\le(1+\e^{c\,t})\e^{\vv\,(1+|\xi(0)|^2)}
+\ff{1}{2}\,\E\Big(\sup_{0\le s\le
t}\e^{\vv\,\e^{-\gg s}(1+|Y^{\xi(0)}(s)|^2)}\Big)\\
&\quad+64\|\si\|_{\rm HS}^2\vv^2\int_0^t\e^{-\gg
s}\E(\e^{\vv\,\e^{-\gg\,
s}(1+|Y^{\xi(0)}(s)|^2)}|Y^{\xi(0)}(s)|^2)\d s,~~
\gg>\beta+2\,\|\si\|_{\rm HS}^2\,\vv.
\end{split}
\end{equation}
So plugging \eqref{s2} back into \eqref{eq9} yields that
\begin{equation}\label{eq10}
\E\Big(\sup_{0\le t\le T}\e^{\vv\,\e^{-\gg
T}|Y^{\xi(0)}(t)|^2}\Big)<\8,~~~\gg>\beta+2\,\|\si\|_{\rm HS}^2\vv.
\end{equation}
Note that
\begin{equation*}
\sup_{\vv>0}(\vv\e^{-(\bb+2\|\si\|_{\rm HS}^2\,\vv)
T})=\ll_T:=\ff{1}{2\|\si\|_{\rm HS}^2T}\e^{-(\bb T+1)}.
\end{equation*}
Consequently, in \eqref{eq10}, by taking
$\gg\downarrow\bb+\ff{1}{T}$, we arrive at
\begin{equation}\label{s6}
\E\Big(\sup_{0\le t\le T}\e^{\ll_0
|Y^{\xi(0)}(t)|^2}\Big)<\8,~~~~\ll_0\in(0,\ll_T)
\end{equation}
In the end, \eqref{eq11} follows from \eqref{s5} and \eqref{s6} in
 case of $\ll T<\ll_T$.
\end{proof}

\begin{rem}
In terms of Lemma \ref{novikov}, \eqref{eq11} holds for small  $T>0$
 provided that \eqref{eq3} is non-dissipative, i.e., $\bb\ge0$ in
\eqref{eq13}. Also, \eqref{eq11} is satisfied with  large $T>0$ in
case that \eqref{eq3} is dissipative, i.e., $\beta<0$ in
\eqref{eq13}.
\end{rem}

For notation brevity, we set
\begin{equation}\label{w5}
 h^\xi_1(t):=\si^{-1}\{b(Y^\xi(t))-b(Y^\xi(t_\dd))-Z(\hat Y_{t_\dd}^\xi
)\},~~~t\ge0,~~\xi\in\C,
\end{equation}
where $\hat Y^\xi_\cdot$ is defined exactly as in \eqref{w6} with
$X^{(\dd)}$ replaced by $Y^\xi$.

The lemma below plays an important  role in checking the Novikov
condition so that the Girsanov theorem is applicable and
investigating weak error analysis.

\begin{lem}\label{le1}
{\rm  Suppose that ({\bf A1}) and ({\bf A2}) hold. Then,
\begin{equation}\label{eq12}
\E\,\e^{\ll\int_0^T|\si^{-1}Z(Y_t^\xi)|^2\d t}<\8
\end{equation}
whenever $\ll,T>0$ such that
\begin{equation*}
\ll<\ff{\e^{-(1+\beta T)}}{2\,\|\si\|_{\rm HS}^2 \|\si^{-1}\|_{\rm
HS}^2\{L_2^2{\bf1}_{\{\aa=1\}}+0{\bf1}_{\{\aa\in(0,1)\}}\}T^2},
\end{equation*}
where we set $\ff{1}{0}=\8$. Moreover,
\begin{equation}\label{w12}
\E\,\e^{\ll\int_0^{T}|h^\xi_1(t)|^2\d t}<\8
\end{equation}
provided that $\ll,T>0$ such that
\begin{equation*}
\ll<\ff{\e^{-(1+\beta T)}}{4\,\|\si\|_{\rm HS}^2 \|\si^{-1}\|^2_{\rm
HS}\{(4L_1^2+ L_2^2){\bf 1}_{\{\aa=1\}}+L_1^2{\bf
1}_{\{\aa\in(0,1)\}}\}T^2}.
\end{equation*}
}
\end{lem}
\begin{proof}
From ({\bf A2}), it is obvious to see that
\begin{equation}\label{w7}
|Z(\xi)|\le |Z(\xi_0)|+L_2\|\xi\|_\8^\aa,~~~~~\xi\in\C,
\end{equation}
which, in addition to  Young's inequality, implies that
\begin{equation}\label{eq5}
|\si^{-1}Z(Y_t^\xi)|^2\le c_\vv+\|\si^{-1}\|_{\rm
HS}^2\{(1+\vv)L_2^2{\bf1}_{\{\aa=1\}}+\vv{\bf1}_{\{\aa\in(0,1)\}}\}\|\xi\|_\8^2,~~~\vv>0
\end{equation}
for some constant $c_\vv>0.$
As a consequence, \eqref{eq12}  holds true from \eqref{eq5}  and by
taking advantage of \eqref{eq11} followed by   choosing
$\vv\in(0,1)$ sufficiently small.

By the definition of $\hat Y_\cdot^\xi$ (see \eqref{w6} with
$X^{(\dd)}$ replaced by $Y^\xi$ for more details), a straightforward
calculation shows that
\begin{equation}\label{w8}
\begin{split}
\|\hat Y^\xi_{t_\dd}\|_\8&=\sup_{-\tau\le
\theta\le0}|\hat Y^\xi_{t_\dd}(\theta)|\\
&\le\max_{k\in\S}\sup_{-(k+1)\dd\le
\theta\le-k\dd}\Big(\ff{\theta+(1+k)\dd}{\dd}|Y^\vv(t_\dd-k\dd)|-\ff{\theta+k\dd}{\dd}|Y^\vv(t_\dd-(k+1)\dd)|\Big)\\
&\le \|Y^\xi_t\|_\8\vee\|Y^\xi_{t-\tau}\|_\8,~~~t\ge0
\end{split}
\end{equation}
due to the fact that $ (\theta+(1+k)\dd)/\dd-(\theta+k\dd)/\dd=1. $
Subsequently, \eqref{w8}, together with ({\bf A1}) as well as
\eqref{w7}, yields that
\begin{equation}\label{eq17}
|h^\xi_1(t)|^2
\le\mu_\vv+\nu_\vv
(\|Y^\xi_t\|_\8^2\vee\|Y^\xi_{t-\tau}\|_\8^2),~~\vv>0,~t\ge0
\end{equation}
for some $\mu_\vv>0$ and
\begin{equation*}
\nu_\vv:=2\|\si^{-1}\|^2_{\rm HS}\{(4L_1^2+
(1+\vv)L_2^2){\bf1}_{\{\aa=1\}}+L_1^2(1+\vv){\bf1}_{\{\aa\in(0,1)\}}\}
\end{equation*}
Thereby, \eqref{w12}   follows from \eqref{s6} and \eqref{eq17}  and
by noting that
\begin{equation*}
\int_0^T\e^{\ll(\|Y^\xi_t\|_\8^2\vee\|Y^\xi_{t-\tau}\|_\8^2)}\d
t\le\tau\|\xi\|_\8^2+2\int_0^T\e^{\ll\|Y^\xi_t\|_\8^2}\d t,~~~\ll>0,
\end{equation*}
where we set $\xi(\theta):=\xi(-\tau)$ for any
$\theta\in[-2\tau,-\tau]$.
\end{proof}

Next we intend to show that the  displacement of segment process is
continuous in the sense of $L^p$-norm sense.

\begin{lem}\label{con}
{\rm Under ({\bf A1}) and ({\bf A3}), for any $p>2$ and $T>0,$ there
exists a constant $C_{p,T}>0$ such that
\begin{equation}\label{s3}
\sup_{0\le t\le T}\E\|Y_t^\xi-\hat Y^\xi_{t_\dd}\|_\8^p\le
C_{p,T}\,\dd^{(p-2)/2}.
\end{equation}
}
\end{lem}

\begin{proof}
By \cite[Theorem 4.4, p61]{M08}, for any $p>0$ and $T>0$,  there
exists $\hat C_{p,T}>0$ such that
\begin{equation}\label{s1}
\E\Big(\sup_{-\tau\le t\le T}|Y^\xi(t)|^p\Big)\le \hat
C_{p,T}(1+\|\xi\|_\8^p).
\end{equation}
By utilizing H\"older's inequality and BDG's inequality, it follows
from ({\bf A1}) and \eqref{s1} that
\begin{equation}\label{y3}
\begin{split}
&\E\Big(\sup_{k\dd \le t \le  k+2 \dd
}|Y^\xi(t)-Y^\xi(k\dd)|^p\Big)\\&\le c\,\Big\{
\dd^{p-1}\int_{k\dd}^{ (k+2)\dd}\E|b(Y^\xi(t))|^p\d t +
\E\Big(\sup_{0 \le t \le
2\dd}|W(t)|^p\Big) \Big\}\\
&\le c\,\Big\{ \dd^{p-1}\int_{k\dd}^{ (k+2)\dd}(1+\E|Y^\xi(t)|^p)\d
t
+ \dd^{p/2}  \Big\}\\
&\le c \,\dd^{p/2},~~~~p>2,~k\in\mathbb{N}.
\end{split}
\end{equation}
 Trivially,  there exists an integer $k_0$ such that
$t\in[k_0\dd, (k_0+1)\dd]$. So, for any $p>2,$
\begin{equation*}
\begin{split}
\E\|Y_t^\xi-\hat Y^\xi_{t_\dd}\|_\8^p &\le
M\max_{k\in\S}\E\Big(\sup_{-(k+1)\dd\le
v\le-k\dd}|Y^\xi(t+\theta)-\hat Y^\xi_{k_0\dd}(\theta)|^p\Big)\\
&\le c\,M\max_{k\in\S}\E
|Y^\xi((k_0-k)\dd)-Y^\xi((k_0-k-1)\dd)|^p\\
&\quad+c\,M\max_{k\in\S}\E\Big(\sup_{(k_0-k-1)\dd\le
s\le(k_0-k+1)\dd}|Y^\xi(s)-Y^\xi((k_0-k-1)\dd)|^p\Big).
\end{split}
\end{equation*}
In case of $k\le k_0-1,$  we find  from \eqref{y3} that \eqref{s3}
holds. On the other hand, if $k=k_0$, from ({\bf A3}), \eqref{y3}
and $M\dd=\tau$,
 then one gets that \eqref{s3} holds. Moreover, for $k\ge 1+k_0,$
\eqref{s3} is still true due to ({\bf A3}). The proof is therefore
complete.
\end{proof}

With the previous Lemmas in hand, we are now in the position to
complete the
\begin{proof}[{\bf Proof of Theorem \ref{th1}}]
Let
\begin{equation}\label{w23}
  W_2^\xi(t)=W(t)+\int_0^th^\xi_1(s)\d
  s,~~~~~t\ge0,
\end{equation}
where $h^\xi_1$ was introduced in \eqref{w5}. Define
\begin{equation*}
R^\xi_2(t)=\exp\Big(-\int_0^t\<h^\xi_1(s),\d
W(s)\>-\ff{1}{2}\int_0^t|h^\xi_1(s)|^2\d s\Big), ~~~~t\ge0
\end{equation*}
and $\d \mathbb{Q}_2^\xi=R_2^\xi(T)\d \P,$ where $T>0$ such that
\eqref{eq19}. Due to  \eqref{eq19} and \eqref{w12},  the Girsanov
theorem implies that $(W_2^\xi(t))_{t\in[0,T]}$ is a Brownian motion
under the probability measure $\mathbb{Q}_2^\xi.$ Thus, \eqref{eq3}
can be rewritten in the following form
\begin{equation}\label{q1}
\d Y^\xi(t) =\{b(Y^\xi(t_\dd))+Z(\hat Y_{t_\dd}^\xi )\}\d t+\si\,\d
W_2^\xi(t),~~~~t>0
\end{equation}
with the initial value $Y^\xi(\theta)=\xi(\theta),
\theta\in[-\tau,0]$ so that $(Y^\xi(t),W_2^\xi(t))_{t\in[0,T]}$ is a
weak solution to \eqref{eq2} under $\mathbb{Q}_2^\xi$. Obviously,
\eqref{eq2} has a unique strong solution so as to the weak solution
is unique. Since, by \eqref{eq19} and \eqref{eq12},
$(Y^\xi(t),W_1^\xi(t))_{t\in[0,T]}$ is a weak solution to
\eqref{eq1} under $\mathbb{Q}_1^\xi$ and
$(Y^\xi(t),W_2^\xi(t))_{t\in[0,T]}$ is a weak solution to
\eqref{eq2} under $\mathbb{Q}_2^\xi$, we deduce from the weak
uniqueness due to Lemma \ref{weak} and H\"older's inequality that
\begin{equation}\label{w9}
\begin{split}
&|\E f(X(t))-\E
f(X^{(\dd)}(t))|\\&=|\E_{\mathbb{Q}_1^\xi}f(Y^\xi(t))-\E_{\mathbb{Q}_2^\xi}f(Y^\xi(t))|\\
&=|\E((R_1^\xi(T)-R_2^\xi(T))f(Y^\xi(t)))|\\
&\le \interleave  f
\interleave_\8\E|R_1^\xi(T)-R_2^\xi(T)|\\
&\le \interleave  f \interleave_\8\E\bigg(
(R_1^\xi(T)+R_2^\xi(T))\Big(\Big|\int_0^t\<\si^{-1}Z(Y_s^\xi)+h^\xi_1(s),\d
W(s)\>\Big|\\
&\quad+
\frac{1}{2}\int_0^t|\,|h^\xi_1(s)|^2-|\si^{-1}Z(Y_s^\xi)|^2|\d
s\Big)\bigg)\\
&\le\interleave  f \interleave_\8
\Big((\E(R_1^\xi(T))^q)^{1/q}+(\E(R_2^\xi(T))^q)^{1/q}\Big)\\
&\quad\times\bigg\{\bigg(\E\Big(\Big|\int_0^t\<\si^{-1}Z(Y_s^\xi)+h^\xi_1(s),\d
W(s)\>\Big|^p\Big)\bigg)^{1/p}\\
&\quad+
\frac{1}{2}\int_0^t(\E|\,|h^\xi_1(s)|^2-|\si^{-1}Z(Y_s^\xi)|^2|^p)^{1/p}\d
s\bigg\}\\
&=:\|f\|_\8\Gamma(T)\{\Theta_1(t)+\Theta_2(t)\},~~~~~t\in[0,T]
\end{split}
\end{equation}
for $1/p+1/q=1, p,q>1,$ where in the second inequality we utilized
the fundamental inequality:
\begin{equation*}
|\e^x-\e^y|\le(\e^x+\e^{y})|x-y|,~~~x,y\in\R,
\end{equation*}
and, in the last two procedure,  employed the Minkowski inequality.
Let
\begin{equation*}
M_1(t)=\int_0^t\<\si^{-1}Z(Y_s^\xi),\d W(s)\> ~~\mbox{ and
}~~M_2(t)=-\int_0^t\<h^\xi_1(s),\d W(s)\>,~~~t\ge0.
\end{equation*}
For any $q>1,$ using H\"older's inequality and the fact that
$\e^{2qM_i(t)-2q^2\<M_i\>(t)}, i=1,2,$  is an exponential martingale
leads to
\begin{equation*}
\begin{split}
&\E(R_1^\xi(T))^q+\E(R_2^\xi(T))^q\\&=\E\,\e^{qM_1(T)-\ff{q}{2}\<M_1\>(T)}+\E\,\e^{qM_2(T)-\ff{q}{2}\<M_2\>(T)}\\
&\le(\E\,\e^{(2q^2-q)\<M_1\>(T)})^{1/2}+(\E\,\e^{(2q^2-q)\<M_2\>(T)})^{1/2}\\
&\le2\bigg(\E\exp\Big((2q^2-q)\int_0^T|\si^{-1}Z(Y_t^\xi)|^2\d
t\Big)\bigg)^{1/2}+\bigg(\E\exp\Big((2q^2-q)\int_0^T|h^\xi_1(t)|^2\d
t\Big) \bigg)^{1/2}.
\end{split}
\end{equation*}
Whence,  by taking $q\downarrow1$ and exploiting \eqref{eq19},
\eqref{eq12},   and \eqref{w12},
 one has, for some $\tilde
C_{q,T}>0$,
\begin{equation}\label{w13}
\Gamma(T)\le \tilde C_{q,T}.
\end{equation}
In view of  ({\bf A1}) and ({\bf A2}), in addition to
$|Y^\xi(t)-Y^\xi(t_\dd)|\le \|Y_t^\xi-\hat Y_{t_\dd}^\xi \|_\8$, it
holds that
\begin{equation}\label{w0}
\begin{split}
|\si^{-1}Z(Y_t^\xi)+h^\xi_1(t)|&\le c\,
\{|b(Y^\xi(t))-b(Y^\xi(t_\dd))|+|Z(Y_t^\xi)-Z(\hat Y_{t_\dd}^\xi )|\}\\
&\le c\,\{L_1|Y^\xi(t)-Y^\xi(t_\dd)|+L_2\|Y_t^\xi-\hat Y_{t_\dd}^\xi
\|_\8^\aa\}\\
&\le c\,\{\|Y_t^\xi-\hat Y_{t_\dd}^\xi \|_\8+\|Y_t^\xi-\hat
Y_{t_\dd}^\xi \|_\8^\aa\}.
\end{split}
\end{equation}
This, besides BDG's inequality followed by H\"older's inequality,
yields that
\begin{equation}\label{w14}
\begin{split}
\Theta_1(t)&\le c\,\bigg(
\int_0^t\E|\si^{-1}Z(Y_s^\xi)+h^\xi_1(s)|^p\d
 s   \bigg)^{1/p}\\
 &\le c\,\bigg(  \int_0^t\{\E\|Y_s^\xi-\hat Y_{s_\dd}^\xi \|_\8^p+ \E\|Y_t^\xi-\hat Y_{s_\dd}^\xi
\|_\8^{p\aa}\}\d
 s   \bigg)^{1/p}\\
 &\le c\,  \dd^{\ff{\aa}{2}-\ff{1}{p}},~~~~p>2/\aa,
\end{split}
\end{equation}
where we utilized \eqref{s3} in the last display. On the other hand,
applying H\"older's inequality and combining ({\bf A1}) with ({\bf
A2}) and \eqref{w0} enables us to obtain that
\begin{equation}\label{w15}
\begin{split}
\Theta_2(t)&\le\ff{1}{2}\int_0^t\{(\E|h^\xi_1(s)-\si^{-1}Z(Y_s^\xi)|^{p/(p-1)})^{p-1}\E|\si^{-1}Z(Y_s^\xi)+h^\xi_1(s)|^p\}^{1/p}\d s\\
&\le c\int_0^t\{(1+\E\|Y_s^\xi\|_\8^p+\E\|\hat
Y_{s_\dd}^\xi\|_\8^p)(\E\|Y_s^\xi-\hat Y_{s_\dd}^\xi
\|_\8^p+\E\|Y_s^\xi-\hat
Y_{s_\dd}^\xi \|_\8^{p\aa})\}^{1/p}\d s\\
&\le c\int_0^t(\E\|Y_s^\xi-\hat Y_{s_\dd}^\xi
\|_\8^p+\E\|Y_s^\xi-\hat
Y_{s_\dd}^\xi \|_\8^{p\aa})\}^{1/p}\d s\\
&\le c \, \dd^{\ff{\aa}{2}-\ff{1}{p}}, ~~~~p>\ff{2}{\aa},
\end{split}
\end{equation}
where we used \eqref{w8}
 and \eqref{s1} in the  penultimate  procedure and  exploited \eqref{s3} in the last step.
Consequently, substituting  \eqref{w13}, \eqref{w14} and \eqref{w15}
into \eqref{w9} and taking $p>2/\aa$ sufficiently large (so that
$q\downarrow1$) yields the assertions \eqref{w10}.
\end{proof}

\section{Extensions to Other Scenarios}\label{sec3}
In this section, we intend to extend the approach  to derive Theorem
\ref{th1} and investigate the weak convergence of   other kind of
numerical schemes for path-dependent SDEs with irregular
coefficients.

\subsection{Extension to Truncated EM Scheme}
In this subsection we are still interested in \eqref{eq1}. Rather
than the EM scheme \eqref{eq2}, we introduce the following truncated
EM scheme associated with \eqref{eq1}
\begin{equation}\label{w1}
\d X^{(\dd)}(t)=\{b(X^{(\dd)}(t_\dd))+Z(\hat X_t^{(\dd)})\}\d
t+\si\,\d W(t),~~~t>0
\end{equation}
with the initial value $X^{(\dd)}(\theta)=X(\theta)=\xi(\theta),
\theta\in[-\tau,0],$ where $\hat X_t^{(\dd)}\in\C$ is defined in the
 way  $$\hat X_t^{(\dd)}(\theta):=X^{(\dd)}((t+\theta)\wedge t_\dd),
\theta\in[-\tau,0].$$

As for the truncated EM scheme \eqref{w1}, the main result in this
subsection is stated as below.
\begin{thm}\label{th2}
Let $({\bf A1})$ and $({\bf A2})$   hold. Then, for any  $T>0$ such
that
\begin{equation*}
2\,\|\si\|_{\rm HS}^2 \|\si^{-1}\|^2_{\rm HS}\{(4L_1^2+ L_2^2){\bf
1}_{\{\aa=1\}}+L_1^2{\bf1}_{\{\aa\in(0,1)\}}\}< \e^{-(1+\beta
T)}/T^2,
\end{equation*}
 there exists a constant $C_{2,T}>0$ such that
\begin{equation}
|\E f(X(t))-\E f(X^{(\dd)}(t))|\le
C_{2,T}\,\dd^{\aa/2},~~~~f\in\B_b(\R^d),~~t\in[0,T].
\end{equation}
\end{thm}

\begin{proof}
Herein we just list some dissimilarities   since the argument of
Theorem \ref{th2} is parallel to that of Theorem \ref{th1}. Set
\begin{equation*}
 h^\xi_2(t):=\si^{-1}\{b(Y^\xi(t))-b(Y^\xi(t_\dd))-Z(\hat Y_t^\xi
)\},~~~t\ge0,~~\xi\in\C
\end{equation*}
with $$\hat Y_t^\xi(\theta)=Y^\xi((t+\theta)\wedge t_\dd),~
\theta\in[-\tau,0].$$ It is easy to see that
\begin{equation*}
\|\hat Y_t^\xi\|_\8=\sup_{t-\tau\le s\le t}|Y^\xi(s\wedge
t_\dd)|\le\|Y_t^\xi\|_\8.
\end{equation*}
So Lemma \ref{le1} still holds with $h_1^\xi$ replaced by $h_2^\xi$
by virtue of Lemma \ref{novikov}. On the other hand, by ({\bf A1})
and \eqref{s1}, we infer from H\"older's inequality and BDG's
inequality that
\begin{equation}\label{a1}
\begin{split}
\E\|Y_t^\xi-\hat Y_t^\xi\|_\8^p&=\E\Big(\sup_{t-\tau\le s\le
t}|Y^\xi(s)-
Y^\xi(s\wedge t_\dd)|^p\Big)\\
&=\E\Big(\sup_{t-\tau\le s\le t}|Y^\xi(s)- Y^\xi(  t_\dd)|^p{\bf
1}_{\{s\ge t_\dd\}}\Big)\\
&=\E\Big(\sup_{t-\tau\le s\le t}\Big|\int_{t_\dd}^sb(Y^\xi(u)\d
u+\int_{t_\dd}^s \si\d W(s)\Big|^p{\bf 1}_{\{s\ge t_\dd\}}\Big)\\
&\le c\,\Big\{\dd^{p-1}\int_{t_\dd}^t|b(Y^\xi(u)|^p\d
u+\E\Big(\sup_{t_\dd\le s\le t}\Big|\int_{t_\dd}^s \si\d
W(s)\Big|^p\Big)\Big\}\\
&\le c\,\dd^{p/2},~~~p\ge1.
\end{split}
\end{equation}
Having Lemma \ref{le1} with writing $h_2^\xi$ in lieu of $h_1^\xi$
and \eqref{a1} in hand, the proof of Theorem \ref{th2} is therefore
complete by inspecting the argument of Theorem \ref{th1}.
\end{proof}

\begin{rem}\label{rem}
{\rm In terms of Theorems \ref{th1} and \ref{th2}, we conclude that
the truncated EM scheme \eqref{w1} enjoys a better weak convergence
rate than the EM scheme \eqref{eq2}. On the other hand, with regard
to the truncated EM scheme, we drop the assumption ({\bf A3})
  in Theorem \ref{th2}. Furthermore, we point out that the EM scheme
  \eqref{eq2} established via interpolation works merely for
  path-dependent SDEs with finite memory. While the truncated EM
  scheme \eqref{w1} is available for path-dependent SDEs with infinite memory
  as the following subsection demonstrates.
}
\end{rem}

\subsection{Extension to path-dependent SDEs with infinite memory}
As we depicted  in Remark \ref{rem}, one of the advantages of the
truncated EM scheme \eqref{w1} is that it is applicable to
path-dependent SDEs with infinite memory. To proceed, let's
introduce some additional notation. For  a fixed number
 $r\in(0,\8),$ let
\begin{equation*}
\mathscr{C}_r=\Big\{\phi\in
C((-\8,0];\R^d):\|\phi\|_r:=\sup_{-\8<\theta\le0}(\e^{r\theta}|\phi(\theta)|)<\8\Big\},
\end{equation*}
which is a Polish space under the metric induced by $\|\cdot\|_r$.

In this subsection, we focus on the following path-dependent SDE
with infinite memory
\begin{equation}\label{a2}
\d X(t)=\{b(X(t))+Z(X_t)\}\d t+\si\d W(t),~~~t>0,~~~X_0=\xi\in\C_r,
\end{equation}
in which
\begin{enumerate}
\item[$(\bf A2')$]
$Z:\C_r\to\R^d$ is H\"older continuous, i.e., there exist
$\aa\in(0,1]$ and   $L_4>0 $  such that
\begin{equation*}
|Z(\xi)-Z(\eta)|\le L_4\|\xi-\eta\|_r^\aa,~~~~\xi,\eta\in\C_r,
\end{equation*}
\end{enumerate}
and the other quantities are stipulated exactly as in \eqref{eq1}.
Similar to \eqref{w1}, we define the truncated EM scheme associated
with \eqref{a2}  by
\begin{equation}
\d X^{(\dd)}(t)=\{b(X^{(\dd)}(t_\dd))+Z(\hat X_t^{(\dd)})\}\d
t+\si\,\d W(t),~~~t>0
\end{equation}
with the initial datum  $X^{(\dd)}(\theta)=X(\theta)=\xi(\theta),
\theta\in(-\8,0],$ in which $\hat X_t^{(\dd)}\in\C_r$ is designed by
   $$\hat
X_t^{(\dd)}(\theta):=X^{(\dd)}((t+\theta)\wedge t_\dd),
\theta\in(-\8,0].$$

The main result in this subsection is presented as follows.
\begin{thm}\label{th3}
{\rm Assume   assumptions of Theorem \ref{th2} hold with ({\bf A2})
replaced by $(\bf A2')$. Then,   there exists a constant $C_{3,T}>0$
such that
\begin{equation}\label{a4}
|\E f(X(t))-\E f(X^{(\dd)}(t))|\le
C_{3,T}\,\dd^{\aa/2},~~~~f\in\B_b(\R^d),~~t\in[0,T]
\end{equation}
provided that the stepsize $\dd\in(0,1)$ is sufficiently small. }
\end{thm}

\begin{proof}
Since
$$\|Y_t^\xi\|_r\le \|\xi\|_r+\sup_{0\le s\le
t}|Y^\xi(s)|, $$ Lemma \ref{novikov} still holds with $\|\cdot\|_\8$
replaced by   $\|\cdot\|_r.$ Also, \eqref{eq12} holds under the
assumptions ({\bf A1}) and $(\bf A2')$ so that \eqref{a2} has a
unique weak solution by following the argument of Lemma \ref{weak}.
Let
\begin{equation*}
 h^\xi_3(t)=\si^{-1}\{b(Y^\xi(t))-b(Y^\xi(t_\dd))-Z(\hat Y_t^\xi
)\},~~~t\ge0,~~\xi\in\C_r,
\end{equation*}
where $$\hat Y_t^\xi(\theta):=Y^\xi((t+\theta)\wedge t_\dd),~
\theta\in(-\8,0].$$ Clearly, we have
\begin{equation*}
\begin{split}
\|\hat Y^\xi_t\|_r
&=\e^{-rt}\sup_{-\8< s\le t}(\e^{rs}|Y^\xi(s)|{\bf 1}_{\{s\le
t_\dd\}})+\e^{-rt}\sup_{-\8< s\le t}(\e^{rs}|Y^\xi(t_\dd)|{\bf
1}_{\{  t_\dd\le s\}})\\
&\le\e^{-rt}\sup_{-\8< s\le t}(\e^{rs}|Y^\xi(s)|{\bf 1}_{\{s\le
t_\dd\}})+\e^{r\dd}\e^{-rt}\sup_{-\8< s\le
t}(\e^{rt_\dd}|Y^\xi(t_\dd)|{\bf 1}_{\{ t_\dd\le s\}})\\
&\le\e^{r\dd}\|Y_t^\xi\|_r.
\end{split}
\end{equation*}
So \eqref{w12}  with writing $h_3^\xi(t)$ instead of $h_1^\xi(t)$
remains true whenever the stesize $\dd\in(0,1)$ is sufficiently
small. Moreover, by virtue of ({\bf A1}), \eqref{s1}, H\"older's
inequality as well as BDG's inequality, it follows that
\begin{equation*}
\begin{split}
\E\|Y_t^\xi-\hat Y^\xi_t\|_r^p &=\e^{-prt}\E\Big(\sup_{-\8< s\le
t}(\e^{prs}|Y^\xi(s)- Y^\xi(s\wedge
t_\dd)|^p)\Big)\\
&\le\E\Big(\sup_{t_\dd< s\le
t}\Big(\Big|\int^s_{t_\dd}b(Y^{\xi(0)}(s))\d
s+\si(W(s)-W(t_\dd)) \Big|^p\Big)\Big)\\
&\le c\,\dd^{p/2},~~~p\ge2.
\end{split}
\end{equation*}
Afterwards, carrying out a similar argument to derive Theorem
\ref{th1} we obtain the desired assertion \eqref{a4}.
\end{proof}

\begin{rem}
{\rm To the best of knowledge, Theorem \ref{th3} is the first result
upon  weak convergence  for path-dependent SDEs with   infinite
memory and irregular drifts. For path-dependent SDEs with finite
memory, Theorems \ref{th1} and \ref{th2} shows that the weak
convergence order can be achieved for any $\dd\in(0,1).$ However,
concerning path-dependent SDEs with infinite memory, the weak
convergence rate can only  be available whenever the stepsize
$\dd\in(0,1)$ is sufficiently small. This illustrates one of the
essential features  between SDEs with finite memory and SDEs
infinite memory. Moreover, Theorem \ref{th3} further shows the
superiority of the truncated EM scheme \eqref{w1} with contrast to
the EM scheme established by interpolations at discrete-time points.
}
\end{rem}

\section{Weak Convergence: Degenerate Case}\label{sec4}
In the previous sections, we investigate weak convergence of EM
schemes and its variants for non-degenerate path-dependent SDEs with
H\"older continuous drifts. In this section, we are still interested
in the same topic but concerned with a class of degenerate SDE on
$\R^{2d}:=\R^d\times\R^d$
\begin{equation}\label{d1}
\begin{cases}
\d X(t)=\{X(t)+Y(t)\}\d t\\
\d Y(t)=\{b(X(t),Y(t))+Z(X_t, Y_t)\}\d t+\si\d W(t),~~~~t\ge0
\end{cases}
\end{equation}
with the initial datum $(X_0,Y_0)=(\xi,\eta)\in\C^2$, where
$b:\R^{2d}\rightarrow\R^d$, $Z: \C^2\to\R^d$,
$\si\in\mathbb{M}_{\mathrm{non}}^{d}$, and $(W(t))_{t\ge0}$ is a
$d$-dimensional Brownian motion on the probability space
$(\OO,\F,(\F_t)_{t\ge0},\P)$. \eqref{d1} is the so-called stochastic
Hamiltonian systems, which has been  investigated considerably in
\cite{MSH,S94,W14,WZ,Zhang}, to name a few.

Throughout this section, we assume that
\begin{enumerate}
\item[(\bf H1)]$b$ is Lipschitz continuous, that is, there exists
$K_1>0$ such that
\begin{equation}\label{d2}
|b(x_1,y_1)-b(x_2,y_2)|\le
K_1(|x_1-y_1|+|y_1-y_2|),~~(x_1,y_1),(x_2,y_2)\in\R^{2d}
\end{equation}
 and there exist
$\aa,\bb,\ll,C>0 $ and $\gg\in(-\aa\bb,\aa\bb)$   such that
\begin{equation}\label{d3}
\<\aa x+\gg y,x+y\>+\<\bb y+\gg x,b(x,y)\>\le C-\ll
(|x|^2+|y|^2),~~~(x,y)\in\R^{2d}.
\end{equation}

\item[(\bf H2)] $Z$ is H\"older continuous, i.e., there exist
$\aa\in(0,1]$ and $K_2>0$ such that
$$|Z(\xi_1,\eta_1)-Z(\xi_2,\eta_2)|\le K_2(\|\xi_1-\xi_2\|_\8^\aa+\|\eta_1-\eta_2\|_\8^\aa),~~~(\xi_1,\eta_1),(\xi_2,\eta_2)\in\C^2.$$
\end{enumerate}

By carrying out a similar argument to derive Lemma \ref{weak} and
taking advantage of Lemma \ref{le2} below, \eqref{d1} has a unique
weak solution under   ({\bf H1}) and ({\bf H2}). With the assumption
\eqref{d2}, the following reference SDE
\begin{equation}\label{d4}
\begin{cases}
\d U(t)=\{U(t)+V(t)\}\d t\\
\d V(t)= b\,(U(t),V(t)) \d t+\si\d W(t),~~~~t\ge0
\end{cases}
\end{equation}
with the initial data $(U(0), V(0))=(u,v)\in\R^{2d}$ is wellposed.
To emphasize the initial value $(u,v)\in\R^{2d}$, we shall write
$(U^{u,v}(t), V^{u,v}(t))$ instead of $(U (t), V (t))$. Analogously
to \eqref{a6}, we can extend respectively $U(t)$ and $V(t)$ in the
following way:
\begin{equation*}
U^{\xi,\eta}(t)=\xi(t){\bf
1}_{[-\tau,0)}(t)+U^{\xi(0),\eta(0)}(t){\bf
1}_{[0,\8)}(t),~~~t\in[-\tau,\8),~~(\xi,\eta)\in\C^2
\end{equation*}
and
\begin{equation*}
V^{\xi,\eta}(t)=\eta(t){\bf
1}_{[-\tau,0)}(t)+V^{\xi(0),\eta(0)}(t){\bf
1}_{[0,\8)}(t),~~~t\in[-\tau,\8),~~(\xi,\eta)\in\C^2.
\end{equation*}
Let $U_t^{\xi,\eta}$ and $ V^{\xi,\eta}_t$ be the segment process
associated with $U^{\xi,\eta}(t)$ and $V^{\xi,\eta}(t)$,
respectively. Next, the truncated EM scheme corresponding to
\eqref{d1} is given by
\begin{equation*}
\begin{cases}
\d X^{(\dd)}(t)=\{X^{(\dd)}(t)+Y^{(\dd)}(t)\}\d t\\
\d Y^{(\dd)}(t)=\{b(X^{(\dd)}(t_\dd),Y^{(\dd)}(t_\dd))+Z(\hat
X_t^{(\dd)}, \hat Y_t^{(\dd)})\}\d t+\si\d W(t)
\end{cases}
\end{equation*}
with the initial value
$(X^{(\dd)}(\theta),Y^{(\dd)}(\theta))=(X(\theta),Y
(\theta))=(\xi(\theta),\eta(\theta))\in\R^{2d}, \theta\in[-\tau,0],$
where
\begin{equation*}
\hat X_t^{(\dd)}(\theta):= X^{(\dd)}((t+\theta)\wedge t_\dd)
~~\mbox{ and }~~\hat Y_t^{(\dd)}(\theta):=
Y^{(\dd)}((t+\theta)\wedge t_\dd),~~~\theta\in[-\tau,0].
\end{equation*}
Observe that
\begin{equation*}
\d X^{(\dd)}(t)=\{X^{(\dd)}(t)+ (b(X^{(\dd)}(0),Y^{(\dd)}(0))t+Z(\tt
X_t^{(\dd)}, \tt Y_t^{(\dd)}))+\si W(t)\}\d t,~~~t\in[0,\dd]
\end{equation*}
where, for any $\theta\in[-\tau,0],$
\begin{equation*}
\Lambda(t):=\int_0^tZ(\tt X_s^{(\dd)},\tt Y_s^{(\dd)})\d s~~\mbox{
with }\tt X_s^{(\dd)}(\theta)=X((t+\theta)\wedge0),~ ~\tt
Y_t^{(\dd)}(\theta):=Y((t+\theta)\wedge0).
\end{equation*}
Thus, $(X^{(\dd)}(t))_{t\in[0,\dd]}$ can be obtained explicitly via
the variation-of-constants formula. Inductively, $X^{(\dd)}(t)$
enjoys explicit formula.

In the sequel,   for $\aa,\bb,\gg$ such that \eqref{d3}, consider
the following Lyapunov function
\begin{equation*}
\mathbb{W}(x,y):=\ff{\aa}{2}|x|^2+\ff{\bb}{2}|y|^2+\gg\<x,y\>,~~~x,y\in\R^d.
\end{equation*}
For  $\gg\in(-\aa\bb,\aa\bb),$ it is easy to see that
\begin{equation}\label{d5}
\kk_2(|x|^2+ |y|^2)\le \mathbb{W}(x,y)\le\kk_1(|x|^2+ |y|^2),
~~~x,y\in\R^d,
\end{equation}
in which $ \kappa_1:= (1+\aa)(1+\bb)/2 $ and
\begin{equation}\label{d7}
\kk_2:=\ff{1}{2}\bigg\{
\bigg(\aa-\ff{1}{2}(\aa/|\gg|+|\gg|/\bb)\bigg)\wedge\Big(\bb-\ff{2|\gg|}{\aa/|\gg|+|\gg|/\bb}\Big)\bigg\}.
\end{equation}

The main result in this section is presented as follows.
\begin{thm}\label{th4}
{\rm Assume ({\bf H1}) and ({\bf H2}) hold. Then, for any $T>0$ such
that
\begin{equation*}
2\,\kk_3\,\|\si\|_{\rm HS}^2 \|\si^{-1}\|_{\rm
HS}^2\{(4K_1^2+K_2^2){\bf1}_{\{\aa=1\}}+2K_1^2{\bf1}_{\{\aa\in(0,1)\}}\}T^2<\kk_2\,\e^{\ll\kk_2
T-1}
\end{equation*}
 there exists $C_{4,T}>0$ such that
\begin{equation}\label{d12}
|\E f(X(t),Y(t))-\E f(X^{(\dd)}(t),Y^{(\dd)}(t))|\le
C_{4,T}\,\dd^{\aa/2},~~~~f\in\B_b(\R^{2d}),~~t\in[0,T].
\end{equation}
}
\end{thm}

\begin{rem}
{\rm The dissipative condition \eqref{d3} is imposed to guarantee
that the time horizontal $T>0$ in Theorem \ref{th4} is large in
certain situation. Nevertheless, in case of $\ll<0$,  \eqref{d12}
remains true but for small time horizontal. Moreover, we can also
investigate weak convergence of EM scheme via interpolation for
\eqref{d1} but with an additional assumption put on the initial
value. Also we point out that, whenever the numerical scheme of the
second component is established by interpolation, the algorithm for
the first component is much more explicit compared with the
truncated EM scheme.

}
\end{rem}

The proof of Theorem \ref{th4} is based on several lemmas below. The
following lemma shows exponential integrability of segment process.
\begin{lem}\label{le2}
{\rm Let \eqref{d3} hold. Then, for any $T>0,$
\begin{equation}\label{d11}
\E\exp\Big(\ll\int_0^T(\|U_t^{\xi,\eta}\|_\8^2+\|V_t^{\xi,\eta}\|_\8^2)\d
t\Big)<\8, ~~~\ll<\ff{\kk_2\,\e^{\ll\kk_2 T-1}}{\kk_3\,\|\si\|_{\rm
HS}^2 T^2}
\end{equation}
with $\kk_3:=\gg^2\vee\bb^2.$}
\end{lem}

\begin{proof}
For notation simplicity, in what follows we write $U(t)$ and $V(t)$
in lieu of $U^{\xi,\eta}(t)$ and $V^{\xi,\eta}(t)$, respectively. By
a close inspection of the proof for Lemma \ref{novikov}, to verify
\eqref{eq11} it is sufficient to show that,  for any $\vv>0$ and
$\gg>-\ll\,\kk_2+\kk_3\,\|\si\|_{\rm HS}^2\,\vv$,
\begin{equation}\label{s7}
\E\Big(\sup_{0\le t\le T}\e^{\vv\,\kk_2\,\e^{-\gg\,
T}(|U(t)|^2+|V(t)|^2)}\Big)<\8,
\end{equation}
where $\kk_2$ was given in \eqref{d7} and $\kk_3:=\gg^2\vee\bb^2.$
By the It\^o formula, it follows from \eqref{d3} and \eqref{d5} that
\begin{equation*}
\begin{split}
&\d (\e^{-\gg t}\mathbb{W}(U(t),V(t)))\\&=\e^{-\gg t}\Big\{-\gg
\mathbb{W}(U(t),V(t))+\<\aa U(t)+\gg V(t),U(t)+V(t)\>\\
&\quad+\<\gg U(t)+\bb V(t),b(U(t),V(t))\> +(C+\|\si\|_{\rm
HS}^2/2)\Big\}\d t+\e^{-\gg t}\<\si^*(\gg U(t)+\bb
V(t)),\d W(t) \>\\
&\le\e^{-\gg t}\Big\{-(\gg+\ll\kk_2) \mathbb{W}(U(t),V(t))
+(C+\|\si\|_{\rm HS}^2/2)\Big\}\d t+\e^{-\gg t}\<\si^*(\gg U(t)+\bb
V(t)),\d W(t) \>.
\end{split}
\end{equation*}
This implies via It\^o's formula that
\begin{equation}\label{a7}
\begin{split}
&\d\e^{\vv\,\e^{-\gg
t}\mathbb{W}(U(t),V(t))}\\&\le-\vv\,(\gg+\ll\kk_2-\kk_3\|\si\|_{\rm
HS}^2\vv)\e^{-\gg
t}\e^{\vv\,\e^{-\gg t}\mathbb{W}(U(t),V(t))}\mathbb{W}(U(t),V(t))\d t\\
&\quad+c_\vv\,\e^{\vv\,\e^{-\gg t}\mathbb{W}(U(t),V(t))}\d
t+\vv\,\e^{-\gg t}\e^{\vv\,\e^{-\gg
t}\mathbb{W}(U(t),V(t))}\<\si^*(\gg U(t)+\bb V(t)),\d W(t)
\>,~~\vv>0
\end{split}
\end{equation}
for some constant $c_\vv>0$. For any
$\gg>-\ll\kk_2+\kk_3\|\si\|_{\rm HS}^2\vv$,   Gronwall's inequality,
in addition to \eqref{d5},  yields that
\begin{equation}\label{d9}
\E\,\e^{\vv\,\e^{-\gg t}W(U(t),V(t))}\le\e^{c_\vv
t}\e^{\vv\kk_1(|\xi(0)|^2+|\eta(0)|^2)},
\end{equation}
which, together with \eqref{a7}, leads further to
\begin{equation}\label{d8}
\begin{split}
&\vv\,(\gg+\ll\kk_2-\kk_3\|\si\|_{\rm HS}^2\vv)\int_0^t\e^{-\gg
s}\,\E\,\e^{\vv\,\e^{-\gg
s}\mathbb{W}(U(s),V(s))}\mathbb{W}(U(s),V(s))\d
s\\
&\le(1+\e^{c_\vv t})\e^{\vv\kk_1(|\xi(0)|^2+|\eta(0)|^2)}.
\end{split}
\end{equation}
Subsequently, by means of BDG's inequality, we derive from
\eqref{d5} and \eqref{d8} that
\begin{equation}\label{d10}
\begin{split}
&\E\Big(\sup_{0\le s\le t}\int_0^s\e^{-\gg u}\e^{\vv\,\e^{-\gg
u}\mathbb{W}(U(u),V(u))}\<\si^*(\gg U(u)+\bb V(u)),\d W(u) \>\Big)\\
&\le 4\ss2\,\E\Big(\int_0^t\e^{-2\gg s}\e^{2\vv\,\e^{-\gg
s}\mathbb{W}(U(s),V(s))}|\si^*(\gg U(s)+\bb V(s))|^2\d s\Big)^{1/2}\\
&\le \ff{1}{2}\,\E\Big(\sup_{0\le s\le t}\e^{ \vv\,\e^{-\gg
s}\mathbb{W}(U(s),V(s))}\Big)+c\,\int_0^t\e^{-\gg
s}\,\E\,\e^{\vv\,\e^{-\gg
s}\mathbb{W}(U(s),V(s))}(|U(s)|^2+|V(s)|^2)\d s\\
&\le\ff{1}{2}\,\E\Big(\sup_{0\le s\le t}\e^{ \vv\,\e^{-\gg
t}\mathbb{W}(U(t),V(t))}\Big)+c\,\int_0^t\e^{-\gg
s}\,\E\,\e^{\vv\,\e^{-\gg
s}\mathbb{W}(U(s),V(s))}\mathbb{W}(U(s),V(s))\d s\\
&\le\ff{1}{2}\,\E\Big(\sup_{0\le s\le t}\e^{ \vv\,\e^{-\gg
t}\mathbb{W}(U(t),V(t))}\Big)+c\,(1+\e^{c_\vv
t})\e^{\vv\kk_1(|\xi(0)|^2+|\eta(0)|^2)}.
\end{split}
\end{equation}
With \eqref{a7}-\eqref{d10} in hand, we thus arrive at
\begin{equation*}
\E\Big(\sup_{0\le t\le T}\e^{\vv\,\e^{-\gg\,
T}\mathbb{W}(U(t),V(t))}\Big)<\8.
\end{equation*}
This, combining with \eqref{d5}, yields \eqref{s7}.
\end{proof}

For notation brevity, we  set
\begin{equation*}
h^{\xi,\eta}(t):=\si^{-1}\{b\,(U^{\xi,\eta}(t),V^{\xi,\eta}(t))-b\,(U^{\xi,\eta}(t_\dd),V^{\xi,\eta}(t_\dd))-Z(\hat
U^{\xi,\eta}_t,\hat V^{\xi,\eta}_t\}.
\end{equation*}

\begin{lem}\label{le3}
{\rm Let ({\bf H1}) and ({\bf H2}) hold. Then,
\begin{equation}\label{r1}
\E\e^{\ll\int_0^T|\si^{-1}Z(U^{\xi,\eta}_t,V^{\xi,\eta}_t)|^2\d
t}<\8
\end{equation}
for any $\ll,T>0$ such that
\begin{equation*}
\ll< \ff{\kk_2\,\e^{\ll\kk_2 T-1}}{2\,\kk_3\,\|\si\|_{\rm HS}^2
\|\si^{-1}\|_{\rm
HS}^2\{K_2^2{\bf1}_{\{\aa=1\}}+0{\bf1}_{\{\aa\in(0,1)\}}\}T^2}
\end{equation*}
Furthermore,
\begin{equation}\label{r7}
\E\,\e^{\ll\int_0^T|h^{\xi,\eta}(t)|^2\d t}<\8,~~~~
\ll<\ff{\kk_2\,\e^{\ll\kk_2 T-1}}{4\,\kk_3\,\|\si\|_{\rm HS}^2
\|\si^{-1}\|_{\rm HS}^2(4K_1^2+K_2^2)T^2}
\end{equation}
for any $\ll,T>0$ such that
\begin{equation*}
\ll<\ff{\kk_2\,\e^{\ll\kk_2 T-1}}{4\,\kk_3\,\|\si\|_{\rm HS}^2
\|\si^{-1}\|_{\rm
HS}^2\{(4K_1^2+K_2^2){\bf1}_{\{\aa=1\}}+2K_1^2{\bf1}_{\{\aa\in(0,1)\}}\}T^2}.
\end{equation*}
}
\end{lem}

\begin{proof}
From ({\bf A2}), it holds that there exists   some constant $c_\vv>0
$ such that, for any $\vv>0,$
\begin{equation}\label{r3}
|\si^{-1}Z(U^{\xi,\eta}_t,V^{\xi,\eta}_t)|^2\le
c_\vv+\{2K_2^2\|\si^{-1}\|_{\rm HS}^2(
1+\vv){\bf1}_{\{\aa=1\}}+\vv{\bf1}_{\{\aa\in(0,1)\}}\}(\|U^{\xi,\eta}_t\|_\8^2+\|V^{\xi,\eta}_t\|_\8^2).
\end{equation}
Henceforth, \eqref{r1}  follows  from \eqref{r3} and Lemma
\ref{le2}.

Next, with the aid of \eqref{d2} and ({\bf H2}) and due to the facts
that $\|\hat U^{\xi,\eta}_t\|_\8\le \| U^{\xi,\eta}_t\|_\8$ and
$\|\hat V^{\xi,\eta}_t\|_\8\le \| V^{\xi,\eta}_t\|_\8$, it follows
that
\begin{equation}\label{r5}
|h^{\xi,\eta}(t)|^2 \le c_\vv +4\,\|\si^{-1}\|_{\rm
HS}^2(4\,K_1^2+K_2^2( 1+\vv)\}(\| U^{\xi,\eta}_t\|_\8^2+\|
V^{\xi,\eta}_t\|_\8^2)
\end{equation}
for some $c_\vv>0$ and
\begin{equation*}
\nu_\vv:=4\,\|\si^{-1}\|_{\rm HS}^2\{(4\,K_1^2+K_2^2(
1+\vv)){\bf1}_{\{\aa=1\}}+2K_1^2{\bf1}_{\{\aa\in(0,1)\}}\}
\end{equation*}
Therefore, by virtue of \eqref{r5}   and Lemma \ref{le2}, \eqref{r7}
  holds true.
\end{proof}

Hereinafter, we proceed to finish the
\begin{proof}[{\bf Proof of Theorem \ref{th4}}]
Under the assumption ({\bf H1}), it is standard to show that
\begin{equation*}
\E\Big(\sup_{-\tau\le t\le
T}(|U^{\xi,\eta}(t)|^p+|V^{\xi,\eta}(t)|^p)\Big)\le
C_{p,T}(\|\xi\|_\8^p+\|\eta\|_\8^p).
\end{equation*}
This, combining H\"older's inequality  with BDG's inequality, leads
to
\begin{equation}\label{r9}
\sup_{0\le t\le T}\E\|U^{\xi,\eta}_t-\hat U^{\xi,\eta}_t
\|_\8^p+\E\|V^{\xi,\eta}_t-\hat V^{\xi,\eta}_t \|_\8^p\le
c\,\dd^{p\,\aa/2}.
\end{equation}
Thus, mimicking the argument of Theorem \ref{th1} we obtain the
desired assertion from \eqref{r9} and Lemma \ref{le3}.
\end{proof}

\section{Weak Convergence: Integrability Conditions}
In the previous sections, we investigated weak convergence of EM
schemes for path-dependent SDEs, where the irregular drifts are at
most linear growth. In this section we still focus on the topic upon
  weak convergence but for path-dependent SDEs under integrability
conditions, which might allow that the irregular drifts need not to
be
 linear growth.

  We start with some additional notation. Denote $ C^2(\R^d)$
 by the set of all continuously twice differentiable functions
 $f:\R^d\rightarrow\R$ and  $C_0^\8(\R^d)$ by the family of arbitrarily often differentiable functions $ f:\R^d\to\R$ with compact support. Let $\nn$ and
$\nn^2$ mean  the gradient operator and the Hessian operator,
respectively. Let $\mathscr{P}(\R^d)$ stand for the collection of
all probability measures on $\R^d.$  For
$\si\in\mathbb{M}_{\mathrm{non}}^{d}$ and $V\in C^2(\R^d)$ with
$\e^{-V}\in L^1(\d x)$ and $\mu_0(\d x):=C_V\e^{-V(x)}\d
x\in\mathscr{P}(\R^d), $ where $C_V$ is the normalization, set
  $Z_0:\R^d\rightarrow\R^d$ by
\begin{equation}\label{50}
Z_0(x):= -(\si\si^*) \nabla V(x),~~~~x\in\R^d.
\end{equation}
Thus, by the integration by parts formula, the operator
\begin{equation*}
\mathscr{L}_0f(x):=\frac{1}{2}\mbox{tr}((\si\si^*)\nn
^2f)(x)+\<Z_0(x),\nn f(x)\>,~~~~x\in\R^d,~~f\in C_0^\8(\R^d)
\end{equation*}
is symmetric on $L^2(\mu_0)$, i.e., for any $f,g\in C_0^\8(\R^d)$,
\begin{equation*}
\mathscr{E}_0(f,g):=\<f,\mathscr{L}_0g\>_{L^2(\mu_0)}=\<g,\mathscr{L}_0f\>_{L^2(\mu_0)}=-\<\si^*\nn
f,\si^*\nn g\>_{L^2(\mu_0)}.
\end{equation*}
Let $H_\si^{1,2}$ be the completion of $C_0^\8(\R^d)$ under
the Sobolev norm
\begin{equation*}
\|f\|_{H_\si^{1,2}}:=(\mu_0(|f|^2+|\si^*f|^2))^{1/2}.
\end{equation*}
Then, $(\mathscr{E}_0,H_\si^{1,2})$ is a symmetric Dirichlet form on
$L^2(\mu_0)$ and the associated Markov process can be constructed as
the solution to the following reference SDE
\begin{equation}\label{51}
\d Y(t)=Z_0(Y(t))\d t+\si \d W(t),~~~~t>0,~~~~Y(0)=x,
\end{equation}
where $W(t)$ is a $d$-dimensional Brownian motion defined on the
probability space $(\OO,\F,\P)$ with the filtration
$(\F_t)_{t\ge0}.$ Assume  that
\begin{enumerate}
\item[(\bf C1)] $Z_0:\R^d\rightarrow\R^d$ is Lipschitz continuous,
i.e., there exist  an $L_0$ such that
\begin{equation*}
|Z_0(x)-Z_0(y)|\le L_0|x-y|£¬~~~~x,y\in\R^d,
\end{equation*}
and there exists constants $C>0$ and $\bb\in\R$ such that
\begin{equation*}
2\<x,Z_0(x)\>\le C+\bb|x|^2,~~~x\in\R^d.
\end{equation*}

\end{enumerate}
Under ({\bf C1}), \eqref{51} has a unique solution
$(Y^x(t))_{t\ge0}$ with the initial value $Y^x(0)=x.$ Observe that
$\mu_0$ is the invariant probability  measure of the Markov
semigroup $P_tf(x):=\E f(Y^x(t))$,  $f\in\B_b(\R^d)$, the set of all
bounded measurable functions on  $\R^d.$

In this section, we consider the following path-dependent SDE
\begin{equation}\label{52}
\d X(t)=\Big\{Z_0(X(t))+\int_{-\tau}^0Z(X(t+\theta))\rho(\d
\theta)\Big\}\d t+\si\d W(t),~~~t\ge0,~~X_0=\xi,
\end{equation}
 where $\rho(\cdot)$ is a probability measure on $[-\tau,0].$ Under
the assumption \eqref{54} below, \eqref{52} admits a unique weak
solution by following exactly the argument of Lemma \ref{weak}. The
EM scheme associated with \eqref{52} is given by
\begin{equation}
\d X^{(\dd)}(t)=\Big\{Z_0(X^{(\dd)}(t_\dd))+\int_{-\tau}^0 Z(\hat
X^{(\dd)}_t(\theta))\rho(\d \theta)\Big\}\d t+\si\d W(t)
\end{equation}
with the initial value $X^{(\dd)}(\theta)=X(\theta)=\xi(\theta),
\theta\in[-\tau,0],$ where
\begin{equation*}
\hat
 X^{(\dd)}_t(\theta):= X^{(\dd)}((t+\theta)\wedge t_\dd),~
 ~\theta\in[-\tau,-0].
\end{equation*}
Analogously, we define
\begin{equation*}
\hat Y^\xi_t(\theta)= Y^\xi((t+\theta)\wedge t_\dd),~
~\theta\in[-\tau,-0],
\end{equation*}
where $Y^\xi$ was extended as in \eqref{a6}. Moreover, we set
\begin{equation*}
h_4^\xi(t):=\si^{-1}\Big\{ Z_0(Y^\xi(t)) - Z_0(Y^\xi(t_\dd))
-\int_{-\tau}^0  Z(\hat Y^\xi_t( \theta))\rho(\d \theta)\Big\}.
\end{equation*}

One of our main result in this section is as follows, which reveals
the weak convergence order of EM scheme for path-dependent SDEs
under an integrability condition.

\begin{thm}\label{th5}
{\rm Assume ({\bf C1}) holds and suppose further that there exists
$\kk>0$ such that
\begin{equation}\label{54}
\mu_0(\e^{\kk |Z(\cdot)|^2})<\8
\end{equation}
and that there exist $m\ge1,\aa\in(0,1]$ and $C>0$ such that
\begin{equation}\label{e8}
|Z(x)-Z(y)|\le C(1+|x|^m+|y|^m)|x-y|^\aa£¬~~~x,y\in\R^d.
\end{equation}
Then, there exists $C_{5,T}>0$ such that
\begin{equation}
|\E f(X^\xi(t))-\E f(X^{(\dd)}(t))|\le
C_{1,T}\,\dd^{\aa},~~\xi\in\C,~~f\in\B_b(\R^d),~~t\in[0,T]
\end{equation}
for any $T>0$ such that
\begin{equation}\label{e9}
1<\ff{\kk}{ 2(2\vee  d )\|\si^{-1}\|_{\rm
HS}^2T^2}\wedge\ff{\e^{-(1+\beta T)}}{32\,\|\si\|_{\rm HS}^2
\|\si^{-1}\|_{\rm HS}^2T^2}\wedge\ff{\kk}{(1\vee \ff{d}{2})T}.
\end{equation}
}
\end{thm}

\begin{proof}
From $\xi\in\C$ and \eqref{e8}, we infer from Lemma \ref{le4} bleow
that \eqref{e5} is available so that
\begin{equation}\label{e4}
\E\,\e^{(1+\vv)\int_0^T|\int_{-\tau}^0Z(Y^\xi(t+\theta))
\rho(\d\theta)|^2\d t}+\E\,\e^{(1+\vv)\int_0^T |h_4^\xi(t)|^2
 \d t}<\8
\end{equation}
for some $\vv\in(0,1)$ sufficiently small and $T>0$ such that
\eqref{e9}. Next, exploiting H\"older's inequality and taking
advantage of \eqref{s1}, \eqref{a1}, and \eqref{e8} enables us to
obtain that
\begin{equation}\label{e12}
\begin{split}
&\int_{-\tau}^0\E|Z(Y^\xi(t+\theta))-Z(\hat
Y^\xi_t(\theta))|^p\rho(\d\theta)\\
&\le\int_{-\tau}^0\E\Big(\sup_{t_\dd\le s\le t}|Z(Y^\xi(s))-Z(
Y^\xi(t_\dd))|^p{\bf1}_{\{t+\theta\ge t_\dd\}}\Big)\rho(\d\theta)\\
&\le c\,E\Big(\sup_{t_\dd\le s\le
t}(1+ |Z(Y^\xi(s))|^{pm}+|Y^\xi(t_\dd)|^{pm})|Y^\xi(s)-Y^\xi(t_\dd)|^{p\aa}\Big)\\
&\le c\,\dd^{p\aa/2}.
\end{split}
\end{equation}
With \eqref{e4} and \eqref{e11} in hand, the proof of Theorem
\ref{th5} can be done by following the line of Theorem \ref{th1}.
\end{proof}

\begin{rem}
The integrability condition \eqref{54} is explicit and verifiable
since the density of $\mu_0$ is explicitly given.  If $\mu_0$ is a
Gaussian measure (e.g., $V(x)=c\,|x|^2$ for some constant $c>0$) and
$Z:\R^d\to\R^d$ is H\"older continuous with the H\"older exponent
$\aa\in(0,1)$, then \eqref{54} holds definitely for any $\kk>0$.
Moreover, the linear growth of $Z$ imposed in Lemma \ref{novikov} is
an essential ingredient, whereas the integrability condition
\eqref{54} does not impose any growth condition on $Z$ and even allows $Z$ to be singular at certain setup, e.g., $Z(x)=(\log
\ff{1}{|x|^\aa}){\bf 1}_{\{|x|\le 1\}}+x{\bf 1}_{\{|x|> 1\}}$,
$x\in\R,$ for some $\aa\in(0,1)$.
\end{rem}

Via the dimension-free Harnack inequality (see e.g. \cite{W10,W11}),
we can establish the following exponential integrability under an
integrability condition, which is an essential ingredient in
analyzing weak convergence.
\begin{lem}\label{le4}
{\rm Assume that ({\bf C1}) and \eqref{54} hold. 
Then, for any $\ll,T>0$ such that
\begin{equation}\label{e4}
\E\,\e^{\ll\int_0^T|\int_{-\tau}^0Z(Y^\xi(t+\theta))
\rho(\d\theta)|^2\d t}<\8,~~~~\ll<\ff{\kk}{(1\vee \ff{d}{2})T}
\end{equation}
and
\begin{equation}\label{e6}
 \E\,\e^{\ll\int_0^T |h_4^\xi(t)|^2
 \d t}<\8
\end{equation}
whenever  $\ll,T>0$ such that
\begin{equation*}
 \ll<\ff{\kk}{2(2\vee  d
)\|\si^{-1}\|_{\rm HS}^2T^2}\wedge\ff{\e^{-(1+\beta
T)}}{32\,\|\si\|_{\rm HS}^2 \|\si^{-1}\|_{\rm
HS}^2L_0^2T^2}.
\end{equation*} }
\end{lem}

\begin{proof}
By H\"older's inequality and Jensen's inequality, it follows that
\begin{equation}
\begin{split}
\E\,\e^{\ll\int_0^T|\int_{-\tau}^0Z(Y^\xi(t+\theta))
\rho(\d\theta)|^2\d t}&\le \E\,\e^{\ll\int_0^T
\int_{-\tau}^0|Z(Y^\xi(t+\theta))|^2\rho(\d \theta) \d t}\\
&\le \ff{1}{T}\int_0^T\int_{-\tau}^0\E\,\e^{\ll T
|Z(Y^\xi(t+\theta))|^2}\rho(\d \theta)\d t\\
&\le\ff{1}{T} \Big\{\int_{-\tau}^0 \e^{\ll T |Z(\xi(\theta))|^2}\d
\theta+\int_0^T\E\,\e^{\ll T |Z(Y^{\xi(0)}(t))|^2}\d t\Big\}.
\end{split}
\end{equation}
If, for any $\gg>0$ and $p>(1\vee d/2)$ with $p\gg<\kk$, there
exists a continuous positive  function $x\mapsto \Lambda_p(x)$ such
that
\begin{equation}\label{e3}
\E\,\e^{\gg |Z(X^x(t))|^2}\le
\Lambda_p(x)(1-\e^{-L_0t})^{-d/2p}(\e^{\kk|Z(\cdot)|^2})^{1/p},
\end{equation}
 then \eqref{e4} holds true 
due to the facts that $1-\e^{-L_0t}\sim L_0 t$ as $t\rightarrow0$
and
\begin{equation*}
\int_0^ts^{-d/2p}\d s<\8,~~~~p>\ff{d}{2}.
\end{equation*}
In what follows, it remains to verify that \eqref{e3} holds.
According to \cite[Theorem 3.2]{W10} (see also \cite[Theorem
1.1]{W11}), the following dimension-free Harnack inequality
\begin{equation}\label{53}
(P_tf(x))^p\le P_tf^p(x)
\exp\bigg(\ff{pL_0|x-y|^2}{2(p-1)(1-\e^{-L_0t})}\bigg),~~~~x,y\in\R^d,~~f\in\B_b(\R^d),~~p>1
\end{equation}
holds.   For any $n,\gg>0$ and $p>(1\vee d/2)$ with $p\gg<\kk$,
applying the Harnack inequality \eqref{53} to the function $\R^d\ni
x\mapsto \e^{\gg|Z(x)|^2}\wedge n\in\B_b(\R^d)$ yields that
\begin{equation*}
\exp\bigg(-\ff{pL_0|x-y|^2}{2(p-1)(1-\e^{-L_0t})}\bigg)\Big(\E\,(\e^{\gg|Z(Y^x(t))|^2}\wedge
n)\Big)^p\le \E\Big(\e^{p\gg|Z(Y^y(t))|^2}\wedge
n^p\Big),~~~~x,y\in\R^d.
\end{equation*}
Thereby, integrating w.r.t. $\mu_0(\d y)$ on both sides and taking
the invariance of $\mu_0$ and \eqref{54} into consideration leads to
\begin{equation*}
\begin{split}
&\exp\Big(-\ff{pL_0}{2(p-1)}\Big)\int_{|x-y|^2\le1-\e^{-L_0t}
}\mu_0(\d y)(\E\,\e^{\gg|Z(Y^x(t))|^2}\wedge
n)^p\\
&\le\int_{\R^d}\E\Big(\e^{p\gg|Z(Y^y(t))|^2}\wedge n^p\Big)\mu_0(\d
y)\\
&\le\mu_0(\e^{p\gg|Z(\cdot)|^2}\wedge
n^p)\le\mu_0(\e^{\kk|Z(\cdot)|^2})<\8,~~~~x\in\R^d,~~p\gg<\kk.
\end{split}
\end{equation*}
 So, by the dominated
convergence theorem, we arrive at
\begin{equation}\label{e1}
\begin{split}
 \Big(\int_{|x-y|^2\le1-\e^{-L_0t} }\mu_0(\d
y)\Big)^{1/p}\E\,\e^{\gg|Z(Y^x(t))|^2}
\le(\mu_0(\e^{\kk|Z(\cdot)|^2}))^{1/p}\exp\Big(\ff{
L_0}{2(p-1)}\Big).
\end{split}
\end{equation}
Next,  from $\mu_0(\d y)=\e^{-V(y)}\d y$ and Taylor's expansion, we
deduce  that
\begin{equation}\label{e2}
\begin{split}
 \int_{|x-y|^2\le1-\e^{-L_0t} }\mu_0(\d y)&= \int_{|x-y|^2\le1-\e^{-L_0t} }\e^{-V(y)}\d
 y\\
 &\ge\e^{-V(x)}\int_{|z|^2\le1-\e^{-L_0t} }\e^{-\int_0^1|\nn V(x+\theta z)|\cdot|z|\d \theta }\d
 z\\
 &\ge\e^{-V(x)}\inf_{|y|\le 1+|x|}\e^{-|\nn V|(y)}\int_{|z|^2\le1-\e^{-L_0t} }\e^{-|z| }\d
 z\\
 &\ge\ff{\pi^{d/2}}{\GG(1+d/2)}\e^{-(1+V(x))}\inf_{|y|\le 1+|x|}\e^{-|\nn
 V|(y)}(1-\e^{-L_0t})^{d/2},
\end{split}
\end{equation}
where $\GG(\cdot)$ is the Gamma function.
 Thence, inserting \eqref{e2} back
into \eqref{e1} gives \eqref{e3}.

A direct calculation shows from ({\bf C1}) and H\"older's inequality
that
\begin{equation*}
|h_4^\xi(t)|^2\le 2\|\si^{-1}\|_{\rm HS}^2\Big\{2L_0^2(|
Y^\xi(t)|^2+ |Y^\xi(t_\dd)|^2)+\int_{-\tau}^0| Z(\hat
Y^\xi_t(\theta))|^2\rho(\d \theta)\Big\}.
\end{equation*}
Thus, H\"older's inequality implies that
\begin{equation*}
\begin{split}
\E\,\e^{\ll\int_0^T|h_4^\xi(t)|^2\d t}&\le\Big(\E\,\e^{16\ll
L_0^2\|\si^{-1}\|_{\rm HS}^2\int_0^T\|Y_t^\xi\|_\8^2\d
t}\Big)^{1/2}\Big(\E\,\e^{4\ll\|\si^{-1}\|_{\rm
HS}^2\int_0^T\int_{-\tau}^0|Z(\hat Y^\xi_t(\theta))|^2\rho(\d
\theta)\d t}\Big)^{1/2}\\
&=:\ss{I_1(T)}\times\ss{I_2(T)}.
\end{split}
\end{equation*}
On one hand, in view of \eqref{eq11}, it holds that
\begin{equation}\label{e7}
I_1(T)<\8,~~~~\ll<\ff{\e^{-(1+\beta T)}}{32\,\|\si\|_{\rm HS}^2
\|\si^{-1}\|_{\rm HS}^2L_0^2T^2}.
\end{equation}
On the other hand, the H\"older inequality and the Jensen inequality
shows that for any $\ll>0,$
\begin{equation*}
\begin{split}
&\E\,\e^{\ll\int_0^T\int_{-\tau}^0|Z(\hat Y^\xi_t(\theta))|^2
\rho(\d\theta)\d t} \\&\le
\ff{1}{T}\int_{-\tau}^0\int_0^T\E\,\e^{\ll T|
Z(\hat Y^\xi_t(\theta))|^2}\d t\rho(\d\theta)\\
&=\ff{1}{T}\ \int_{-\tau}^0\int_0^T\E\,\Big\{\e^{\ll
T|Z(Y^\xi(t+\theta))|^2}{\bf1}_{\{t+\theta\le t_\dd\}}+\e^{\ll
T|Z(Y^\xi(t_\dd))|^2}{\bf1}_{\{t+\theta> t_\dd\}}\Big\}\d t\rho(\d\theta)\\
&\le\ff{1}{T}\ \Big\{\int_{-\tau}^0\e^{\ll T|Z(\xi(\theta))|^2}\d
\theta+\e^{\ll T|Z(\xi(0))|^2}+\int_0^T\E\, \e^{\ll
T|Z(Y^\xi(t))|^2} \d t+ \int_\dd^T\E\,\e^{\ll T|Z(Y^\xi(t_\dd))|^2}
\d t \Big\}
\end{split}
\end{equation*}
so that, by virtue of 
\eqref{e3},
\begin{equation}\label{e11}
I_2(T)<\8,~~~~\ll<\ff{\kk}{2(2\vee  d )\|\si^{-1}\|_{\rm HS}^2T^2}.
\end{equation}
Thus, \eqref{e6} follows 
\eqref{e7} as well as
\eqref{e11} immediately.
\end{proof}

\end{document}